\documentclass[reqno]{amsart}
\usepackage[T1]{fontenc}
\usepackage[utf8x]{inputenc}
\usepackage[all]{xy}
\usepackage{amssymb}
\usepackage{fullpage}
\usepackage{hyperref}

\newcommand{\cE}{\mathcal{E}}
\newcommand{\cZ}{\mathcal{Z}}

\newcommand{\Z}{\mathbb{Z}}
\newcommand{\R}{\mathbb{R}}
\newcommand{\C}{\mathbb{C}}

\newcommand{\PGL}{\mathbf{PGL}}
\newcommand{\PSL}{\mathbf{PSL}}
\newcommand{\U}{\mathbf{U}}
\newcommand{\W}{\mathbf{W}}

\newcommand{\Id}{\mathrm{Id}}

\newcommand{\lmt}{\longmapsto}
\newcommand{\lra}{\longrightarrow}

\newcommand{\Ga}{\Gamma}
\newcommand{\ga}{\gamma}
\newcommand{\si}{\sigma}

\renewcommand{\rho}{\varrho}
\renewcommand{\phi}{\varphi}

\newcommand{\ov}[1]{\overline{#1}}

\newtheorem{theorem}{Theorem}[section]
\newtheorem{proposition}[theorem]{Proposition}
\newtheorem{lemma}[theorem]{Lemma}
\newtheorem{corollary}[theorem]{Corollary}

\theoremstyle{definition}
\newtheorem{definition}[theorem]{Definition}

\newtheorem{remark}[theorem]{Remark}

\numberwithin{equation}{section}

\begin{document}

\title[Parabolic vector bundles on Klein surfaces]{Parabolic vector bundles on Klein surfaces}

\author[I. Biswas]{Indranil Biswas}

\address{School of Mathematics, Tata Institute of Fundamental
Research, Homi Bhabha Road, Mumbai 400005, India}

\email{indranil@math.tifr.res.in}

\author[F. Schaffhauser]{Florent Schaffhauser}

\address{Universidad de Los Andes, Departamento de Matemáticas, Carrera 1 \#18A-12, 111 711 Bogot\'a, Colombia \& Universit\'e de Strasbourg, Institut de Recherche Mathématique Avancée, 7 rue Ren\'e Descartes, 67 000 Strasbourg, France.}

\email{schaffhauser@math.unistra.fr}



\date{\today}

\begin{abstract}
Given a discrete subgroup $\Ga$ of finite co-volume of $\PGL(2,\R)$, we define and study 
parabolic vector bundles on the quotient $\Sigma$ of the (extended) hyperbolic plane by $\Ga$. 
If $\Ga$ contains an orientation-reversing isometry, then the above is equivalent to studying 
real and quaternionic parabolic vector bundles on the orientation cover
of degree two of $\Sigma$. We then 
prove that isomorphism classes of polystable real and quaternionic parabolic vector bundles are 
in a natural bijective correspondence with the equivalence classes of real and quaternionic unitary 
representations of $\Gamma$. Similar results are obtained for compact-type real parabolic 
vector bundles over Klein surfaces.
\end{abstract}

\maketitle

\section{Introduction}

The study of vector bundles over varieties defined over the field of real numbers was initiated
in the seminal work of Atiyah \cite{AtiyahKR}. The present work investigates parabolic vector bundles
over curves defined over $\mathbb R$.

Holomorphic vector bundles on compact Riemann surfaces is a rich topic that started with the 
papers of Weil and Atiyah (\cite{Weil_matrix_div}, \cite{Atiyah_connections}), who gave 
necessary and sufficient conditions for such bundles to arise from linear representations of the 
fundamental group of the Riemann surface over which
the bundle lives. The notion of \textit{stability}, which had been introduced by Mumford in 
order to construct moduli varieties of such bundles (\cite{Mumford_Proc,Mumford_GIT_first_ed}), 
was then used by Narasimhan and Seshadri to prove that vector bundles arising from irreducible 
\textit{unitary} representations were exactly the stable vector bundles of degree zero 
(\cite{NS}).
Seshadri later introduced the notion of parabolic vector bundles (\cite{Se}), and
then proved with Mehta an analog of the Narasimhan-Seshadri theorem for unitary representations
of the fundamental group of a punctured Riemann surface
\cite{MS}. More precisely, they proved that,
given a finite subset $S$ of a compact Riemann surface $X$, irreducible representations of
$\pi_1(X\setminus S)$ in $\U(r)$ correspond bijectively to the stable parabolic vector bundles on
$X$ of rank $r$ and parabolic degree zero with parabolic structure over the points of $S$.

In this paper, we call \textit{Klein surface} a Riemann surface equipped with an anti-holomorphic involution. Just 
as compact Riemann surfaces correspond to smooth complex projective curves, compact Klein surfaces correspond to 
smooth complex projective curves defined over the field of real numbers. Let then $(X,\si)$ be a compact connected 
Klein surface ($\si$ being the given anti-holomorphic involution of the Riemann surface $X$). A real vector bundle 
on $(X,\sigma)$ is a holomorphic vector bundle on $X$ equipped with an anti-holomorphic lift $\tau$ of $\sigma$ 
such that $\tau^2=Id$, and a quaternionic vector bundle on $(X,\sigma)$ is a holomorphic vector bundle on $X$ 
equipped with an anti-holomorphic lift $\tau$ of $\sigma$ such that $\tau^2=-Id$. Real and quaternionic vector 
bundles on Klein surfaces have been intensively investigated (\cite{KW,BHH,Sch_GD}) and our aim here is to study 
the parabolic analog of them. We note that real parabolic vector bundles and linear connections on such bundles 
have already been studied by S. Amrutiya and L. A. Calvo independently in 
\cite{Amrutiya_connections_on_real_parab_vb,Amrutiya_real_parab_vb}, \cite{Calvo} (see
also \cite{GW16} for related results). In the present paper, we will 
focus on stable and polystable such bundles, in both the real and quaternionic case, and on unitary connections on 
such bundles. Given $(X,\sigma)$ as above ($X$ compact) and a finite subset $S\subset X$ such that $\sigma(S)\,=\,S$, 
we set $Y \,:=\, X\setminus S$ and we let $\Ga(Y,\,\si)\,:=\,\pi_1^{\mathrm{orb}}(Y/\si)$ be the orbifold fundamental 
group of the quotient of $Y$ by the action of the group $\langle \si\rangle$ generated by the anti-holomorphic 
involution $\si\vert_Y$. Then there is a short exact sequence
\begin{equation}\label{fund_ses}
1\,\lra\, \pi_1(Y) \,\lra\, \Ga(Y,\,\si) \,\lra\, \Z/2\Z \,\lra\, 1
\end{equation} and, by definition, a \textit{unitary representation} of this short exact sequence will be a group morphism from $\Ga(Y,\,\si)$ to an extension $\W(r)$ of the form 
\begin{equation}\label{extensions_by_unit_gps}
1\longrightarrow \U(r) \longrightarrow \W(r) \longrightarrow \mathbb{Z}/2\mathbb{Z} \longrightarrow 1
\end{equation} which induces the identity on $\Z/2\Z$ (see Definition \ref{rep_of_fund_ses}). Recall that, when $\Z/2\Z$ acts on $\U(r)$ by complex conjugation, one has $H^2(\Z/2\Z;\mathcal{Z}(\U(r)))=\{\pm1\}$. Therefore, with respect to that action, there are only two isomorphism classes of extensions of the form $\eqref{extensions_by_unit_gps}$, the trivial class being that of the semi-direct product $\U(r)\rtimes \Z/2\Z$. Representations of $\Ga(Y,\,\si)$ into the trivial
(respectively, non-trivial) type of extension will be called \textit{real} (respectively, \textit{quaternionic}) unitary representations of $\Ga(Y,\,\si)$.
We then prove the following result.

\begin{theorem}\label{MS_corresp_for_real_and_quat}
Let $(X,\si)$ be a compact Klein surface and let $S\subset X$ be a finite subset such that $\si(S)=S$. Let $Y$ be the punctured surface $X\setminus S$, assumed to be of negative Euler characteristic, and let $\Ga(Y,\,\si)$ be the orbifold fundamental group of the quotient surface $Y/\si$. Then there is a bijective correspondence between equivalence classes of real (respectively, quaternionic) unitary representations of $\Ga(Y,\,\si)$ and isomorphism classes of polystable real (respectively, quaternionic) parabolic vector bundles of parabolic degree zero on $(X,\si)$ with parabolic structure at the points of $S$.
\end{theorem}

More precisely, we prove that every polystable real (respectively, quaternionic) parabolic 
vector bundle of parabolic degree zero on $(X,\si,S)$ is isomorphic to a parabolic bundle of the 
form $\cE(\rho)_\bullet$, where $\rho$ is a real (respectively, quaternionic) unitary 
representation of $\Ga(Y,\,\si)$ and $\cE(\rho)_\bullet$ is equipped with its canonical real 
(respectively, quaternionic) structure (constructed in Proposition 
\ref{construction_of_real_parab_vb}). A similar result is then proved for compact-type real 
parabolic vector bundles (see Definition \ref{compact_type_real_and_quat_parab_bundles}, 
Proposition \ref{construction_of_compact_type_real_parab_vb} and Theorem 
\ref{MS_corresp_for_compact_type_real}).

\section{Parabolic vector bundles and stability}\label{parab_vb_and_stability}

\subsection{Real and quaternionic parabolic vector bundles}

Let $X$ be a compact connected Riemann surface.
Let $\si:X\lra X$ be an anti-holomorphic involution of $X$. 
Fix a finite subset $S=\{x_1,\cdots,x_d\}\subset\ X$
such that $\sigma(S)= S$. Note that $\sigma$ may not fix $S$ point-wise (compare \cite[p.~212]{MS}). Let $E$ be a holomorphic vector bundle on $X$. 

\begin{definition}[{\cite[Definition 1.5]{MS}}]
A \textit{quasi-parabolic structure} on $E$ over $S$ consists of $d$ strictly decreasing filtrations
$$
E_{x_j}\,=E^1_j\supsetneq E^2_j \,\supsetneq\cdots\supsetneq\,
E^{n_j}_j \supsetneq E^{n_j+1}_j \,=0 \quad (1\leq j\leq d),
$$ and a \textit{parabolic structure} on $E$ over $S$ is a
quasi-parabolic structure as above, together with $d$ strictly increasing sequences of real numbers
$$
0\,\leq\,\alpha^1_j\,<\,\alpha^2_j\,<\,
\cdots\,<\, \alpha^{n_j}_j \,<\, 1\ \quad (1\,\leq\, j\,\leq\, d)\, .
$$
The real number $\alpha^i_j$ is called the \textit{parabolic weight} of the subspace $E^i_j$ in
the quasi-parabolic filtration. The \textit{multiplicity} of a parabolic weight
$\alpha^i_j$ at $x_j$ is defined to be the dimension of the complex vector space $E^i_j/E^{i+1}_j$.
A \textit{parabolic vector bundle} on $(X,S)$ is a triple $(E,\{E^i_j\}, \,\{\alpha^i_j\})$ consisting of a vector bundle on $X$ equipped with a parabolic structure over $S$. 
\end{definition}

We refer to \cite{MS} for the definition of a morphism of parabolic vector bundles. For 
notational convenience, a parabolic vector bundle $(E,\{E^i_j\}, \,\{\alpha^i_j\})$ as above 
will also be denoted by $E_{\bullet}$.

Let $E$ be a holomorphic vector bundle on $X$ of rank $r$. Let $\overline{E}$ be the 
$C^\infty$ complex vector bundle of rank $r$ on $X$ whose underlying real vector bundle of 
rank $2r$ is the $C^\infty$ real vector bundle underlying $E$ and whose multiplication by $\sqrt{-1}$
on the fibers coincides with the multiplication by $-\sqrt{-1}$
on the fibers of $E$. The holomorphic structure on $E$ induces a holomorphic
structure on $\sigma^*\overline{E}$, uniquely characterized by the fact that the holomorphic sections of $\si^*\ov{E}$ over U are the anti-holomorphic sections of $E$ on $\si(U)$. Note that a morphism of holomorphic vector bundles $\phi:E\lra F$ gives rise to a morphism of holomorphic vector bundles $\si^*\ov{\phi}:\si^*\ov{E}\lra \si^*\ov{F}$.
Let then $E_{\bullet}\,=(E,\{E^i_j\}, \,\{\alpha^i_j\})$ be a parabolic bundle. For all $x_j\in S$, the filtration $\{E^i_j\}$ of $E_{x_j}$ produces a filtration
of the fiber $(\sigma^*\overline{E})_{\sigma(x_j)}$ using the conjugate linear
identification of $E_{x_j}$ with $(\sigma^*\overline{E})_{\sigma(x_j)}$. For any $j$, the
parabolic weights $\{\alpha^i_j\}_{i=1}^{n_j}$ associated to $\{E^i_j\}_{i=1}^{n_j}$ can be
considered as parabolic weights associated to this filtration of
$(\sigma^*\overline{E})_{\sigma(x_j)}$. The new parabolic vector bundle obtained this way will be denoted by
$\sigma^*\overline{E}_{\bullet}$. We then have the following definition of a real parabolic vector
bundle (equivalent to \cite[Definition 3.1]{Amrutiya_real_parab_vb}; see also \cite[Section~3]{Calvo}), and its immediate generalization to the quaternionic case.

\begin{definition}[Real and quaternionic parabolic vector bundles]
A \textit{real parabolic vector bundle} on $(X,\si,S)$ is a parabolic vector bundle $E_{\bullet}\,=(E,\{E^i_j\}, 
\,\{\alpha^i_j\})$ on $(X,S)$, together with a holomorphic isomorphism of parabolic vector bundles $\phi:E_{\bullet}\lra \si^*\ov{E}_{\bullet}$ such that $\si^*\ov{\phi}\circ \phi = \Id_E$.
A \textit{quaternionic parabolic vector bundle} on $(X,\si,S)$ is a parabolic vector bundle $E_{\bullet}$ on $(X,S)$, together with a holomorphic isomorphism of parabolic vector bundles $\phi:E_{\bullet}\lra \si^*\ov{E}_{\bullet}$ such that $\si^*\ov{\phi}\circ \phi = -\Id_E$.
\end{definition}

Note that a real or quaternionic parabolic vector bundle in particular satisfies $\si^*\ov{E}_\bullet \simeq E_\bullet$. Conversely, if $E_\bullet$ is a \textit{simple} parabolic vector bundle over $(X,\si)$ that satisfies $\si^*\ov{E}_\bullet \simeq E_\bullet$, then $E_\bullet$ admits either a real or a quaternionic structure. Let now $E^{\vee}_{\bullet}$ be the parabolic dual of $E_{\bullet}$ (see \cite[Section~3]{Yo}, \cite[p.~309]{Bi}). The holomorphic vector bundle underlying the parabolic vector bundle $E^{\vee}_{\bullet}$ will be denoted by $E^{\vee}_0$. Note that $E^{\vee}_0$ is in general only a subsheaf of $E^\vee$ (the dual of $E$ as a vector bundle), and that the inclusion map $E^{\vee}_0\,\hookrightarrow\,E^\vee$ is an isomorphism over $X\setminus S$, but fails
to be an isomorphism over $x_j\in S$ if there is a non-zero parabolic weight of $E_{\bullet}$ at $x_j$. Note that $(E_\bullet^\vee)^\vee$ is canonically isomorphic to $E_\bullet$. 
\begin{definition}[Compact-type real parabolic vector bundles]\label{compact_type_real_and_quat_parab_bundles}
A \textit{compact-type real parabolic vector bundle} on $(X,\si,S)$ is a parabolic vector bundle
$E_{\bullet}\,=(E,\{E^i_j\}, \,\{\alpha^i_j\})$ on $(X,S)$, together with a holomorphic isomorphism $\phi:E_\bullet\lra (\si^*\ov{E})_\bullet^\vee$ such that $(\si^*\ov{\phi})^\vee \circ \phi = \Id_E$.
\end{definition}

\begin{remark}\label{on_compact_type_quaternionic} 
Note that a compact-type real parabolic vector bundle in particular satisfies $(\si^*\ov{E})_\bullet^\vee \simeq E_\bullet$. Conversely, if $E_\bullet$ is a \textit{simple} parabolic vector bundle over $(X,\si)$ that satisfies $(\si^*\ov{E})_\bullet^\vee \simeq E_\bullet$, then $E_\bullet$ admits a compact-type real structure (because $\mathrm{Aut}(E_\bullet)\simeq \mathcal{Z}(\mathbf{GL}(r;\C)) \,\simeq\,\C^*$ is $2$-divisible). In principle, one could also define compact-type quaternionic parabolic vector bundles (by requiring that $(\si^*\ov{\phi})^\vee \circ \phi = -\Id_E$), but these can in fact only occur for vector bundles with structure group contained in $\mathbf{SL}(2m;\C)$, as $\mathcal{Z}(\mathbf{SL}(2m;\C))$ is no longer $2$-divisible.
\end{remark}

\subsection{Stability}

The \textit{parabolic degree} of a parabolic vector bundle $E_{\bullet}\,=(E,\{E^i_j\},
\,\{\alpha^i_j\})$ is defined (\cite[Definition~1.11]{MS}) to be $$ \text{par-deg}(E_{\bullet})\,=\text{degree}(E)+\sum_{j=1}^d \sum_{i=1}^{n_j}
\alpha^i_j \dim (E^i_j/E^{i+1}_j).$$ Take any holomorphic subbundle $F\subset E$. For each $x_j\in\,
S$, the fiber $F_{x_j}$ has a filtration obtained by intersecting the quasi-parabolic filtration
of $E_{x_j}$ with the subspace $F_{x_j}$. The parabolic weight of a subspace $V\subset\,
F_{x_j}$ in this filtration is the maximum of the numbers $\{\alpha^i_j\mid V\subset E^i_j\cap F_{x_j}\}$. This parabolic structure on $F$ will be denoted by $F_{\bullet}$.

\begin{definition}[Stability for parabolic bundles]
A parabolic vector bundle $E_{\bullet}\,=(E,\{E^i_j\}, \,\{\alpha^i_j\})$ is called
\textit{stable} (respectively, \textit{semistable}) if for all subbundles $F\subsetneq
E$ of positive rank, the inequality
$$
\frac{\text{par-deg}(F_{\bullet})}{\text{rank}(F)}\,<\,
\frac{\text{par-deg}(E_{\bullet})}{\text{rank}(E)}\ \
\left(\text{respectively, } \ \frac{\text{par-deg}(F_{\bullet})}{\text{rank}(F)} \leq\,
\frac{\text{par-deg}(E_{\bullet})}{\text{rank}(E)}\right)
$$
holds \cite[Definition~1.13]{MS}. A parabolic vector bundle $E_{\bullet}$ is called \textit{polystable} if it is semistable and isomorphic to a direct sum of parabolic stable vector bundles (see \cite{Yo}, \cite{Bi} for
direct sums of parabolic vector bundles).
\end{definition}

A stable parabolic bundle is simple, meaning that any holomorphic automorphism of it is a constant
non-zero scalar multiplication (see \cite{MS}). The parabolic bundles appearing in the direct sum defining a polystable parabolic bundle all have the same parabolic slope, equal to that of $E$.

\begin{definition}[Stability for real and quaternionic parabolic bundles]\label{stability_of_real_or_quat_parab_vb}
A real or quaternionic parabolic vector bundle $(E_{\bullet},\phi) \,=((E,\{E^i_j\}, \,\{\alpha^i_j\}),
\phi)$ is called
\textit{stable} (respectively, \textit{semistable}) if for all subbundles $F\subsetneq
E$ of positive rank with $\phi(F)\,=\sigma^*\overline{F}$, the inequality
$$
\frac{\text{par-deg}(F_{\bullet})}{\text{rank}(F)}\,<\,
\frac{\text{par-deg}(E_{\bullet})}{\text{rank}(E)}\ \
\left(\text{respectively, } \ \frac{\text{par-deg}(F_{\bullet})}{\text{rank}(F)} \leq\,
\frac{\text{par-deg}(E_{\bullet})}{\text{rank}(E)}\right)
$$
holds. A real (respectively, quaternionic) parabolic vector bundle $(E_{\bullet},\phi)$ is called
\textit{polystable} if it is semistable and isomorphic to a direct sum of real
(respectively, quaternionic) parabolic stable vector bundles.
\end{definition}

A stable real parabolic vector bundle is not necessarily simple; take for instance the trivial 
parabolic structure on a real vector bundle of the form $E\,=\,F\bigoplus\si^*\ov{F}$, with $F$ stable 
and $\si^*\ov{F}\,\not\simeq\, F$. However, the automorphism group of such a bundle is either 
$\R^*$ or $\C^*$, as follows from adapting the techniques of \cite{Sch_JSG} to the present 
parabolic setting.

\begin{definition}[Stability for compact-type real parabolic bundles]
A compact-type real 
parabolic vector bundle $(E_{\bullet},\phi) \,=((E,\{E^i_j\}, \,\{\alpha^i_j\}), \phi)$ is called
\textit{stable} (respectively, \textit{semistable}) if for all subbundles $F\subsetneq
E$ of positive rank such that $\phi(F)$ annihilates $\si^*\ov{F}$, the inequality
$$
\frac{\text{par-deg}(F_{\bullet})}{\text{rank}(F)}\,<\,
\frac{\text{par-deg}(E_{\bullet})}{\text{rank}(E)}\ \
\left(\text{respectively, } \ \frac{\text{par-deg}(F_{\bullet})}{\text{rank}(F)} \leq\,
\frac{\text{par-deg}(E_{\bullet})}{\text{rank}(E)}\right)
$$
holds. A compact-type real (respectively, quaternionic) parabolic vector bundle $(E_{\bullet},\tau)$ is called
\textit{polystable} if it is semistable and isomorphic to a direct sum of compact-type real
(respectively, quaternionic) parabolic stable vector bundles.
\end{definition}

We can then state the main result of Section \ref{parab_vb_and_stability}.

\begin{theorem}\label{polystability_over_R_and_over_C}
A real or quaternionic parabolic vector bundle $((E,\{E^i_j\}, \,\{\alpha^i_j\}),
\tau)$ is semistable (respectively, polystable) if and only if the parabolic vector
bundle $(E,\{E^i_j\}, \,\{\alpha^i_j\})$ is semistable (respectively, polystable).

A compact-type real
parabolic vector bundle $((E,\{E^i_j\}, \,\{\alpha^i_j\}),
\tau)$ is semistable (respectively, polystable) if and only if the parabolic vector
bundle $(E,\{E^i_j\}, \,\{\alpha^i_j\})$ is semistable (respectively, polystable).
\end{theorem}

\begin{proof}
Let $(E_\bullet\, ,\phi)$ be a real or quaternionic parabolic vector bundle. If $E_\bullet$ is semistable, then $(E_\bullet\, , \phi)$ is evidently semistable. For the converse, note that if $(E_\bullet\, , \phi)$ is real (respectively, quaternionic), then it follows from the uniqueness of the Harder--Narasimhan filtration of the parabolic vector bundle $E_\bullet$, it follows that this filtration is preserved by $\phi$. Consequently, if $(E_\bullet\, ,\phi)$ is semistable, then $E_\bullet$ is semistable. A similar argument applies in the compact-type case.

If $(E_\bullet\, , \phi)$ is a real or quaternionic parabolic vector bundle such that $E_\bullet$ is polystable, then $$E_\bullet \simeq E^{(1)}_\bullet \oplus \ldots \oplus E^{(k)}_\bullet$$ with $E^{(m)}_\bullet$ stable. The issue is that a given $E^{(m)}_\bullet$ in that decomposition may not be a real (respectively, quaternionic) sub-bundle of $(E_\bullet\, ,\phi)$. But, if it is not, then $\si^*\ov{E^{(m)}}_\bullet$ also appears in the decomposition of $E_\bullet$ as a direct sum of stable parabolic bundles, and it is easy to check, adapting the techniques of \cite{Sch_JSG}, that in this case $E^{(m)}\oplus \si^*\ov{E^{(m)}}$ is a stable real parabolic bundle. Conversely, if $(E_\bullet\, , \phi)$ is polystable, then
we know that $E_\bullet$ is semistable because
$(E_\bullet\, , \phi)$ is semistable, therefore that
$E_\bullet$ has a unique maximal polystable parabolic
subbundle $F$ with parabolic slope the same as that of $E_\bullet$
\cite[page 23, Lemma 1.5.5]{HL}. Now, from the uniqueness
of such a maximal polystable parabolic
subbundle, it follows that $F$ is invariant under $\phi$ (here, $\phi$-invariant means that $\phi(F)\subset\si^*\ov{F}$). As $(E_\bullet\, ,\phi)$ is polystable, the subbundle $F$ has a
$\phi$-invariant complement $F'$. But if $F'$ is non-zero, then the maximality of
$F$ is contradicted, because $F'$ also has a unique maximal polystable parabolic
subbundle $F$ with parabolic slope same as that of $E_\bullet$.
Therefore, we conclude that $F'\,=0$, which implies that
$E_\bullet$ is polystable. Again, a similar argument applies in the compact-type case.
\end{proof}

\begin{corollary}\label{geom_stable_bdles}
A stable real (respectively, quaternionic) parabolic vector bundle on $(X,\si,S)$ is, in general, only polystable as a parabolic vector bundle. If it is in fact stable as a parabolic vector bundle, it will be called geometrically stable as a real (respectively, quaternionic) parabolic vector bundle.

A stable compact-type real
parabolic vector bundle on $(X,\si,S)$ is, in general, only polystable as a parabolic vector bundle. If it is in fact stable as a parabolic vector bundle, it will be called geometrically stable as a compact-type real
parabolic vector bundle.
\end{corollary}

\section{Parabolic vector bundles arising from unitary representations}

\subsection{From unitary representations to polystable parabolic vector bundles}

Let $(X,\si)$ be a compact Klein surface and let $S=\{x_1,\,\ldots\, x_d\}$ be a finite subset of $X$ satisfying $\si(S)=S$. Set $Y:=X\setminus S$ and let $\Ga(Y,\,\si)$ be the orbifold fundamental group of the quotient surface $Y/\si$. Recall that $\Ga(Y,\,\si)$ fits in the fundamental short exact sequence \eqref{fund_ses}. Consider the action of $\Z/2\Z$ on the unitary group $\U(r)$ by complex conjugation and recall that $H^2(\Z/2\Z;\cZ(\U(r)))=\{\pm1\}$, where $\cZ(\U(r))\simeq S^1$ is the center of $\U(r)$. We denote by $\W(r)_\pm$ the corresponding extensions $$1\lra\U(r) \lra \W(r)\pm\lra \Z/2\Z\lra 1.$$ In particular, $\W(r)_+\simeq \U(r)\rtimes\Z/2\Z$.

\begin{definition}[Real and quaternionic unitary representations]\label{rep_of_fund_ses}
A \textit{real unitary representation} of $\Ga(Y,\,\si)$ is a group morphism $\rho\,:\,\Ga(Y,\,\si)
\, \lra \,\W(r)_+$ such that the following diagram commutes.
$$
\xymatrix{
1 \ar[r] & \pi_1(Y) \ar[r] \ar[d] & \Ga(Y,\,\si) \ar[r] \ar[d]^{\rho} & \Z/2\Z \ar[r] \ar@{=}[d] & 1\\
1 \ar[r] & \U(r) \ar[r] & \W(r)_+\ar[r] & \Z/2\Z \ar[r] & 1
}$$
A \textit{quaternionic unitary representation} of $\Ga(Y,\,\si)$ is a group morphism $\rho:\Ga(Y,\,\si) \lra \W(r)_-$ such that the following diagram commutes.
$$
\xymatrix{
1 \ar[r] & \pi_1(Y) \ar[r] \ar[d] & \Ga(Y,\,\si) \ar[r] \ar[d]^{\rho} & \Z/2\Z \ar[r] \ar@{=}[d] & 1\\
1 \ar[r] & \U(r) \ar[r] & \W(r)_-\ar[r] & \Z/2\Z \ar[r] & 1
}$$
Two real (respectively, quaternionic) unitary representations $\rho$ and $\rho'$ are called \textit{equivalent} if there is an element $u\in\U(r)$ such that, for all $\ga\,\in\,\Ga(Y,\,\si)$, we have $u\rho(\ga)u^{-1} = \rho'(\ga)$.
\end{definition}

In particular, if $\rho\,:\,\Ga(Y,\,\si)\,\lra\, \W(r)_\pm$ is either a real or a quaternionic representation, $\rho|_{\pi_1(Y)}:\pi_1(Y)\lra \U(r)$ defines a parabolic vector bundle $\cE(\rho)_{\bullet}$ over $X$ (\cite[Section~1]{MS}). The parabolic weights at $x_j$ and their multiplicities are 
given by the eigenvalues of the image under $\rho$ of an element of $\pi_1(Y)$ produced by an oriented loop around $x_j$ (this element of $\pi_1(Y)$ 
is not unique, but its conjugacy class is, so the eigenvalues and their multiplicities are also uniquely determined). By \cite[Proposition~1.12]{MS}, the parabolic vector bundle 
$\cE(\rho)_{\bullet}$ is polystable of parabolic degree zero, and it is stable if and only if the representation $\rho|_{\pi_1(Y)}$ is irreducible. Conversely, any polystable parabolic vector bundle of rank $r$ and parabolic degree zero is given by a homomorphism from $\pi_1(Y)$ to $\U(r)$ (\cite{MS}).

\begin{proposition}\label{construction_of_real_parab_vb}
If $\rho\,:\,\Ga(Y,\,\si)\,\lra\, \W(r)_+$ is a real unitary representation of $\Ga(Y,\,\si)$, then the parabolic vector bundle $\cE(\rho)_\bullet$, of parabolic degree $0$, has a canonical real structure, with respect to which it is polystable.

If $\rho\,:\,\Ga(Y,\,\si)\,\lra\, \W(r)_-$ is a quaternionic unitary representation of $\Ga(Y,\,\si)$, then the parabolic vector bundle $\cE(\rho)_\bullet$, of parabolic degree $0$, has a canonical quaternionic structure, with respect to which it is polystable.
\end{proposition}

This is proved by going deeper into the details of the construction of the parabolic vector bundle $\cE(\rho)_\bullet$ over $X$. We do it below in the real case, the quaternionic one being similar (see \cite[Section 2]{Sch_JDG}). Note that, in \cite{MS}, the construction of $\cE(\rho)_\bullet$ is presented from a slightly different point of view, closer to that of \cite{Weil_matrix_div}. One starts with a Fuchsian group of finite co-volume $\Ga\subset \PSL(2,\R)$ and consider the compact Riemann surface $X:=(\mathcal{H}\cup\{\mathrm{cusps\ of}\ \Ga\})/\Ga$, where $\mathcal{H}$ is the real hyperbolic plane, acted upon holomorphically by the discrete group $\Ga$. The finite set $S\subset X$ is the projection to $X$ of the fixed points and cusps of $\Ga$ in $\widehat{\mathcal{H}}:=\mathcal{H}\cup\partial\mathcal{H}$. In particular, each point $x\in S$ has an order $1\leq d_x \leq \infty$, which by definition is the order of the cyclic group $\mathrm{Stab}_\Ga(h)$, for all $h$ lying above $x$ in $\widehat{\mathcal{H}}$. So $d_x=\infty$ if and only if $x$ comes from a cusp of $\Ga$ (in $\partial\mathcal{H}$). The relation with the previous point of view is given by the fact that a generator of $\mathrm{Stab}_\Ga(h)$ comes from the homotopy class of a small loop around $x$ in $X\setminus S$. Note that, on the one hand, if $\Ga$ is cocompact, then $X=\mathcal{H}/\Ga$ can be viewed as an orbifold Riemann surface, with $S$ equal to the set of orbifold singularities of $X$. Parabolic vector bundles over $(X,S)$ then correspond to orbifold vector bundles over $X$ (\cite{Boden,Biswas_Duke}). On the other hand, if $\Ga$ is torsion-free of finite co-volume, so has only cusps as fixed points in $\widehat{\mathcal{H}}$, then $\Ga\simeq \pi_1(X\setminus S)$.

\begin{proof}[Proof of Proposition \ref{construction_of_real_parab_vb}]
In view of \cite[Proposition~1.12]{MS} and Theorem \ref{polystability_over_R_and_over_C}, it suffices to proves that the parabolic vector bundle $\cE(\rho)_\bullet\lra X$ has a real (respectively, quaternionic) structure which is compatible with its parabolic structure; it will then automatically be polystable in the real (respectively, quaternionic) parabolic sense. The real (respectively, quaternionic) structure in question comes from the $\Ga(Y,\,\si)$-equivariant structure of the vector bundle $\mathcal{H}\times\C^r$ over $\mathcal{H}$, associated to the representation $\rho\,:\,\Ga(Y,\,\si)\,\lra\,\W(r)_\pm$, which is defined by $\ga\cdot(h,v) := (\ga\cdot h,\rho(\ga)v)$. If $\ga\,\in\,\Ga(Y,\,\si)\setminus\pi_1(Y)$, then $\ga^2$ is, so it descends to an anti-holomorphic transformation $\cE(\rho)=(\mathcal{H}\times\C^r)/\Ga$, covering the real structure of the base. It is then easy to check that this induced transformation squares to $\Id_{\cE(\rho)}$ (respectively, $-\Id_{\cE(\rho)}$) if $\rho$ is a real (respectively, quaternionic) representation, and that it does not depend on the choice of the element $\ga\,\in\,
\Ga(Y,\,\si)\setminus\pi_1(Y)$; see \cite[Section 2]{Sch_JDG} for details. If we now add the fixed points of $\Ga(Y,\,\si)$, then the parabolic structure of $(\widehat{\mathcal{H}}\times\C^r)/\pi_1(Y)$ over $S$ comes, as we have recalled, from the $\mathrm{Stab}_{\pi_1(Y)}(h)$-equivariant structure of $\widehat{\mathcal{H}}\times\C^r$ around $y\in\widehat{\mathcal{H}}$ (\cite[Section 1]{MS}). It therefore suffices to show that this equivariant structure extends to an equivariant structure for the larger stabilizer subgroup $\mathrm{Stab}_{\Ga(Y,\si)}(h)$, which is indeed the case since $\rho$ is, by assumption, a unitary representation of the full $\Ga(Y,\,\si)$.
\end{proof}

\begin{remark}
When $S$ is contained in the fixed point set $X^\si$ (compare \cite[p.~212]{MS}), we can be more explicit in the proof of Proposition \ref{construction_of_real_parab_vb} for real parabolic vector bundles. Indeed, in this case, the parabolic structure on $E_x$ (for all $x\in S$) comes from a representation of the stabilizer group $\mathrm{Stab}_{\pi_1(Y)}(h)\lra\U(r)$ for some chosen $h$ above $x$ in $\widehat{\mathcal{H}}$. Such a representation sends a generator $\ga_y$ of the cyclic group $\mathrm{Stab}_{\pi_1(Y)}(h)$ to the unitary matrix 
$$\begin{pmatrix}
\exp^{2\pi\sqrt{-1} \beta_1} & & 0\\
& \ddots & \\
0 & & \exp^{2\pi\sqrt{-1}\beta_n}
\end{pmatrix}$$ where $(\beta_1,\, \ldots\, ,\beta_n)$ are the (non-necessarily distinct) weights of $\cE(\rho)_\bullet$ at $x$. This representation is $\Z/2\Z$-equivariant with respect to the involution $z\lmt z^{-1}$ on $\mathrm{Stab}_{\pi_1(Y)}(h)$ and complex conjugation on $\U(r)$, so it extends to a group morphism from $\mathrm{Stab}_{\pi_1(Y)}(h)\rtimes \Z/2\Z \simeq \mathrm{Stab}_{\Ga(Y,\si)}(h)$ to $\U(r)\rtimes\Z/2\Z\simeq \W(r)_+$, as claimed.
\end{remark}

We now perform a similar analysis in the compact-type case. This time, we consider the trivial action of $\Z/2\Z$ on the unitary group $\U(r)$. Then we have $H^2(\Z/2\Z;\cZ(\U(r)))=1$, because the Abelian group $\cZ(\U(r))\simeq S^1$ is $2$-divisible (note that this would not happen for $\mathbf{SU}(2m)$ and that this is related to what we were saying in Remark \ref{on_compact_type_quaternionic}, regarding the existence of compact-type quaternionic vector bundles with structure group contained in $\mathbf{SL}(2m;\C)$; see also \cite{Ho}). Therefore, any extension of $\Z/2\Z$ by $\U(r)$ in this case is isomorphic to the direct product $\U(r)\times\Z/2\Z$. In particular, if we look at the analog of Definition \ref{rep_of_fund_ses} in this context, we find that unitary representations of $\Ga(Y,\,\si)$ inducing the identity on $\Z/2\Z$ correspond bijectively to unitary representations of $\Ga(Y,\,\si)$ in the usual sense, i.e.\ group morphisms $\rho\,:\,\Ga(Y,\,\si)\,\lra\,\U(r)$. We then have the following analog of Proposition \ref{construction_of_real_parab_vb}.

\begin{proposition}\label{construction_of_compact_type_real_parab_vb}
Let $\rho\,:\Ga(Y,\,\si)\,\lra\, \U(r)$ be a unitary representation of $\Ga(Y,\,\si)$. Then the parabolic vector bundle $\cE(\rho)_\bullet$ has a canonical compact-type real structure, with respect to which it is polystable.
\end{proposition}

\begin{proof}
It suffices to prove that there is an isomorphism $\phi:\cE(\rho)_\bullet\lra (\si^*\cE(\rho))_\bullet^\vee$ such that $(\si^*\ov{\phi})^\vee\circ\phi = \Id_{\cE(\rho)}$. This follows from the fact that $\mathcal{H}\times \C^r$ descends to a Hermitian vector bundle $\mathcal{V}$ on the open dianalytic surface $Y/\si=\mathcal{H}/\Ga(Y,\,\si)$, which in turn pulls back to an anti-invariant vector bundle $\mathcal{W}$ on the double cover $Y$ of $Z$ (where by anti-invariant we mean that $\si^*\ov{\mathcal{W}}^\vee\simeq \mathcal{W}$ on $Y$). As $\mathcal{V}^\vee\simeq \ov{\mathcal{V}}$ on $Y/\si$, the isomorphism $\si^*\ov{\mathcal{W}}^\vee\simeq \mathcal{W}$ over $Y$ is actually a compact-type real structure, by construction. Finally, $\mathcal{W}$ is just $\cE(\rho)$, since both come from the $\pi_1(Y)$-equivariant vector bundle $\mathcal{H}\times\C^r$, induced on by $\rho:\Ga(Y,\,\si)\,\lra\, \U(r)$. The compatibility with the parabolic structure is checked in the same way as in Proposition \ref{construction_of_real_parab_vb}.
\end{proof}

\subsection{From polystable parabolic vector bundles to unitary representations}

In this section, we prove the main result of this paper (Theorem \ref{MS_corresp_for_real_and_quat}) and its analog for compact-type real bundles (Theorem \ref{MS_corresp_for_compact_type_real}). The strategy is to construct an inverse to the maps $\rho\lmt\cE(\rho)_\bullet$ constructed in Propositions \ref{construction_of_real_parab_vb} and \ref{construction_of_compact_type_real_parab_vb}, thus proving that every polystable real (respectively, quaternionic) parabolic vector
bundle of parabolic degree zero comes from a real (respectively, quaternionic) representation $\rho
\,:\,\Ga(Y,\,\si)\,\lra\,\W(r)_\pm$, and that every polystable compact-type real parabolic vector bundle of parabolic degree zero comes from a unitary representation in the usual sense $\rho\,:
\,\Ga(Y,\,\si)\,\lra\,\U(r)$. Given a polystable real (respectively, quaternionic, respectively, compact-type real) vector bundle $(E_\bullet\, , \phi)$ of parabolic degree zero, such a map is obtained via the holonomy representation of a certain singular unitary connection on $E_\bullet$, namely the Chern connection $A_h$ of an adapted Hermitian-Yang-Mills metric $h$ on $E_\bullet$ (see \cite[Section 2.1]{Biq}). Note that, if we do not impose any compatibility condition on $A_h$ and $\phi$, then we are exactly in the situation of the Mehta-Seshadri theorem (since $E_\bullet$ is polystable in view of Theorem \ref{polystability_over_R_and_over_C}), so this already gives us a holonomy representation $\rho':\pi_1(Y)\lra \U(r)$ such that $\cE(\rho')_\bullet\simeq E_\bullet$ and all that is left to do is incorporate the real (respectively, quaternionic, respectively, compact-type real) structure $\phi$ and show that, when $A_h$ is compatible with $\phi$ in a sense to be made precise, then $\rho'$ extends to a real or quaternionic representation $\rho\,:\,\Ga(Y,\,\si) \,\lra\, \W(r)_\pm$ (respectively, a unitary representation
$\rho\,:\,\Ga(Y,\,\si) \,\lra\, \U(r)$ in the compact-type case), which is what we deal with next, as it has independent interest (Lemma \ref{ext_of_hol_rep}). To simplify the exposition, we focus in what follows on the real case. Unless explicitly stated otherwise, the quaternionic and compact-type real cases are dealt with similarly.

\begin{lemma}\label{ext_of_hol_rep}
Let $(E_\bullet\, , \phi)$ be a real parabolic vector bundle of parabolic degree zero on the compact Klein surface $(X,\si,S)$. Let $h$ be an adapted Hermitian-Yang-Mills metric on $E_\bullet$ and let $A_h$ be the associated Chern connection. If $A_h$ is compatible with $\phi$ in the sense that the pullback of the singular connection $\si^*A_h$ on $\si^*\ov{E}_\bullet$ under the isomorphism $\phi:E_\bullet \lra\si^*\ov{E}_\bullet$ coincides with $A_h$, then the holonomy representation $\rho':\pi_1(X\setminus S)\lra \U(r)$ extends to a group morphism $\rho:\pi_1^{\mathrm{orb}}((X\setminus S)/\si)\lra \U(r)\rtimes\Z/2\Z$.
\end{lemma}

\begin{proof}
Since $E_\bullet$ has parabolic degree $0$ and $h$ is an adapted Hermitian-Yang-Mills metric on $E_\bullet$, the restriction of the Chern connection $A_h$ to $E|_{X\setminus S}$ is flat (this follows for instance from the Gauss-Bonnet formula \cite[Proposition 2.9]{Biq} and the fact that the curvature of the Chern connection of an adapted Hermitian-Yang-Mills metric is constant over $X\setminus S$, by definition of Yang-Mills equations). So we get a holonomy representation $\rho':\pi_1(X\setminus S)\lra\U(r)$. As the connection $A_h$ is compatible with the real structure $\phi$, we can apply the general properties of parallel transport for invariant connections (\cite[Section 4.1]{Sch_JDG}) and show that $\rho'$ extends to a group morphism $\rho:\pi_1^{\mathrm{orb}}((X\setminus S)/\si)\lra \U(r)\rtimes\Z/2\Z$.
\end{proof}

The final step is then to ensure the existence of an invariant adapted Hermitian-Yang-Mills metric on $(E_\bullet\, , \phi)$. Indeed, after that, it is a simple matter to check that the Chern connection of an invariant adapted Hermitian is itself invariant (this follows from the uniqueness of the Chern connection), which puts us in the situation of Lemma \ref{ext_of_hol_rep}. By invariant metric on $(E_\bullet\, , \phi)$, we mean here an adapted Hermitian metric $h$ on $E_\bullet$ such that the isomorphism $\phi:E_\bullet \lra \si^*\ov{E}_\bullet$ becomes an isometry when $\si^*\ov{E}_\bullet$ is equipped with the metric $\si^*\ov{h}$ induced by $h$. Equivalently, we may use $\phi$ to pull back the metric $\si^*\ov{h}$ to $E_\bullet$~: this sets up a $\Z/2\Z$-action on the space of metrics, and an invariant metric is then just a fixed point of that action. Note that this action is induced by the (anti-holomorphic) involution $\si:Y\lra Y$ and the real structure $\phi:E_\bullet\lra \si^*\ov{E}_\bullet$, which may also be seen as a lift $\tau:E_\bullet \lra E_\bullet$ of the real structure $\si:Y\lra Y$. The existence, under the assumption that $(E_\bullet\, , \phi)$ is polystable, of an invariant adapted Hermitian-Yang-Mills metric on $E_\bullet$ then follows from the work of C. Simpson (\cite[Theorem 1 and Corollary 3.6]{Simpson_JAMS_1988} and \cite[Theorem 6]{Simpson_JAMS_1990}). Strictly speaking, in \cite{Simpson_JAMS_1988}, Simpson considers equivariant bundles over $Y$ when the latter is equipped with a finite group of holomorphic transformations, but since anti-holomorphic transformations of $Y$ also act on the set of Hermitian metrics on a complex vector bundle $E$ over $Y$, his proof extends directly to our setting. Other approaches to the existence of adapted Hermitian-Yang-Mills connections on polystable parabolic vector bundles can be found in \cite{Biq,Poritz,Nasatyr_Steer_orbi_approach}, that could presumably also be used to construct invariant Hermitian-Yang-Mills connections on polystable real parabolic vector bundles. Note that one advantage of Simpson's approach is that it does not require using the polystability of the underlying parabolic vector bundle (Theorem \ref{polystability_over_R_and_over_C}), only its polystability in the real sense (also, Simpson considers parabolic Higgs bundles, which are more general than the parabolic vector bundles we study here). We sum up our discussion in the following statement.

\begin{theorem}\label{existence_of_HYM_metrics}
Let $(E_\bullet\, , \phi)$ be a real parabolic vector bundle over the open Klein surface $(Y,\si)$. Then $(E_\bullet\, , \phi)$ is polystable if and only if it admits an invariant, adapted Hermitian-Yang-Mills metric. In particular, if $E_\bullet$ has vanishing parabolic degree, then the Chern connection of such a metric $h$ is flat and invariant. So, by Lemma \ref{ext_of_hol_rep}, its holonomy representation $\rho_h:\pi_1(Y)\lra \U(r)$ extends to a real unitary representation $\widetilde{\rho_h}:\pi_1^{\mathrm{orb}}(Y/\si)\lra \U(r)\rtimes\Z/2\Z$.
\end{theorem}

We have thus proved Theorem \ref{MS_corresp_for_real_and_quat}, as well as the following result for compact-type real parabolic vector bundles.

\begin{theorem}\label{MS_corresp_for_compact_type_real}
Let $(X,\si)$ be a compact Klein surface and let $S\subset X$ be a finite subset such that $\si(S)=S$. Let $Y$ be the punctured surface $X\setminus S$, assumed to be of negative Euler characteristic, and let $\Ga(Y,\,\si)$ be the orbifold fundamental group of the quotient surface $Y/\si$. Then there is a bijective correspondence between conjugacy classes of unitary representations of $\Ga(Y,\,\si)$ and isomorphism classes of polystable compact-type parabolic vector bundles of parabolic degree zero on $(X,\si)$ with parabolic structure at the points of $S$.
\end{theorem}

Finally, the following result follows directly from Theorem \ref{existence_of_HYM_metrics} and the notion of geometrically stable real or quaternionic (respectively, compact-type real) parabolic vector bundle (see Corollary \ref{geom_stable_bdles}).

\begin{theorem}
Geometrically stable real and quaternionic parabolic vector bundles on $(X,\si,S)$ correspond bijectively to real and quaternionic unitary representations $\rho:\pi_1^{\mathrm{orb}}((X\setminus S)/\si) \lra \W(r)_\pm$ such that $\rho|_{\pi_1(X\setminus S)}: \pi_1(X\setminus S) \lra \U(r)$ is irreducible.

Geometrically stable compact-type real parabolic vector bundles on $(X,\si,S)$ correspond bijectively to unitary representations $\rho:\pi_1^{\mathrm{orb}}((X\setminus S)/\si) \lra \U(r)$ such that $\rho|_{\pi_1(X\setminus S)}: \pi_1(X\setminus S) \lra \U(r)$ is irreducible.
\end{theorem}

\section*{Acknowledgments}

The authors wish to thank the referee for helpful comments.
I.~Biswas is supported by a J. C. Bose Fellowship. F.~Schaffhauser was supported
by \textit{Convocatoria 2018-2019 de la Facultad de Ciencias 
(Uniandes), Programa de investigaci\'on “Geometr\'ia y Topolog\'ia de los Espacios de 
M\'odulos”}, the \textit{European Union’s Horizon 2020 research and innovation programme under 
grant agreement No 795222} and the \textit{University of Strasbourg Institute of Advanced Study 
(USIAS)}, and U.S. National Science Foundation grants DMS 1107452, 1107263, 1107367 “RNMS: 
Geometric structures And Representation varieties” (the GEAR Network).


\begin{thebibliography}{Amr14b}

\bibitem[Amr14a]{Amrutiya_connections_on_real_parab_vb}
S.~Amrutiya.
\newblock Connections on real parabolic bundles over a real curve.
\newblock {\em Bull. Korean Math. Soc.}, 51(4):1101--1113, 2014.

\bibitem[Amr14b]{Amrutiya_real_parab_vb}
S.~Amrutiya.
\newblock Real parabolic vector bundles over a real curve.
\newblock {\em Proc. Indian Acad. Sci. Math. Sci.}, 124(1):17--30, 2014.

\bibitem[Ati57]{Atiyah_connections}
M.~F. Atiyah.
\newblock Complex analytic connections in fibre bundles.
\newblock {\em Trans. Amer. Math. Soc.}, 85:181--207, 1957.

\bibitem[Ati66]{AtiyahKR}M.~F. Atiyah.
\newblock K-theory and reality.
\newblock {Quart. Jour. Math. Oxford}, 17:367--386, 1966.

\bibitem[BHH10]{BHH}
I.~Biswas, J.~Huisman, and J.~Hurtubise.
\newblock The moduli space of stable vector bundles over a real algebraic curve.
\newblock {\em Math. Ann.}, 347(1):201--233, 2010.

\bibitem[Biq91]{Biq}
O.~Biquard.
\newblock Fibr\'{e}s paraboliques stables et connexions singuli\`eres plates.
\newblock {\em Bull. Soc. Math. France}, 119(2):231--257, 1991.

\bibitem[Bis97]{Biswas_Duke}
I.~Biswas.
\newblock Parabolic bundles as orbifold bundles.
\newblock {\em Duke Math. J.}, 88(2):305--325, 1997.

\bibitem[Bis03]{Bi}
I.~Biswas.
\newblock On the principal bundles with parabolic structure.
\newblock {\em J. Math. Kyoto Univ.}, 43(2):305--332, 2003.

\bibitem[Bod91]{Boden}
H.~U. Boden.
\newblock Representations of orbifold groups and parabolic bundles.
\newblock {\em Comment. Math. Helv.}, 66(3):389--447, 1991.

\bibitem[Ca18]{Calvo} L. A. Calvo.
\newblock Augmented bundles and real structures.
\newblock Ph.D. Thesis, UAM-CSIC, 2018.

\bibitem[GW16]{GW16} O. Garc\'{\i}a-Prada and G. Wilkin. Action of the mapping class
group on character varieties and Higgs bundles, arXiv:1612.02508.

\bibitem[HL10]{HL}
D.~Huybrechts and M.~Lehn.
\newblock {\em The geometry of moduli spaces of sheaves}.
\newblock Cambridge Mathematical Library. Cambridge University Press,
Cambridge, second edition, 2010.

\bibitem[Ho04]{Ho}
N.-K. Ho.
\newblock The real locus of an involution map on the moduli space of flat
connections on a {R}iemann surface.
\newblock {\em Int. Math. Res. Not.}, 61:3263--3285, 2004.

\bibitem[KW03]{KW}
M.~Karoubi and C.~Weibel.
\newblock Algebraic and {R}eal {$K$}-theory of real varieties.
\newblock {\em Topology}, 42(4):715--742, 2003.

\bibitem[MS80]{MS}
V.~B. Mehta and C.~S. Seshadri.
\newblock Moduli of vector bundles on curves with parabolic structures.
\newblock {\em Math. Ann.}, 248(3):205--239, 1980.

\bibitem[Mum63]{Mumford_Proc}
D.~Mumford.
\newblock Projective invariants of projective structures and applications.
\newblock In {\em Proc. {I}nternat. {C}ongr. {M}athematicians ({S}tockholm,
1962)}, pages 526--530. Inst. Mittag-Leffler, Djursholm, 1963.

\bibitem[Mum65]{Mumford_GIT_first_ed}
D.~Mumford.
\newblock {\em Geometric invariant theory}.
\newblock Ergebnisse der Mathematik und ihrer Grenzgebiete, Neue Folge, Band
34. Springer-Verlag, Berlin-New York, 1965.

\bibitem[NS65]{NS}
M.~S. Narasimhan and C.~S. Seshadri.
\newblock Stable and unitary vector bundles on a compact {R}iemann surface.
\newblock {\em Ann. of Math. (2)}, 82:540--567, 1965.

\bibitem[NS95]{Nasatyr_Steer_orbi_approach}
E.~B. Nasatyr and B.~Steer.
\newblock The {N}arasimhan-{S}eshadri theorem for parabolic bundles: an
orbifold approach.
\newblock {\em Philos. Trans. Roy. Soc. London Ser. A}, 353(1702):137--171, 1995.

\bibitem[Por93]{Poritz}
J.~A. Poritz.
\newblock Parabolic vector bundles and {H}ermitian-{Y}ang-{M}ills connections
over a {R}iemann surface.
\newblock {\em Internat. J. Math.}, 4(3):467--501, 1993.

\bibitem[Sch11]{Sch_GD}
F.~Schaffhauser.
\newblock Moduli spaces of vector bundles over a {K}lein surface.
\newblock {\em Geom. Dedicata}, 151(1):187--206, 2011.

\bibitem[Sch12]{Sch_JSG}
F.~Schaffhauser.
\newblock Real points of coarse moduli schemes of vector bundles on a real algebraic curve.
\newblock {\em J. Symplectic Geom.}, 10(4):503--534, 2012.

\bibitem[Sch17]{Sch_JDG}
F.~Schaffhauser.
\newblock On the {N}arasimhan-{S}eshadri correspondence for {R}eal and
{Q}uaternionic vector bundles.
\newblock {\em J. Differential Geom.}, 105(1):119--162, 2017.

\bibitem[Se77]{Se} C. S. Seshadri.
\newblock Moduli of vector bundles on curves with parabolic structures.
\newblock {\em Bull. Amer. Math. Soc.}, 83:124--126, 1977.

\bibitem[Sim88]{Simpson_JAMS_1988}
C.~T. Simpson.
\newblock Constructing variations of {H}odge structure using {Y}ang-{M}ills
theory and applications to uniformization.
\newblock {\em J. Amer. Math. Soc.}, 1(4):867--918, 1988.

\bibitem[Sim90]{Simpson_JAMS_1990}
C.~T. Simpson.
\newblock Harmonic bundles on noncompact curves.
\newblock {\em J. Amer. Math. Soc.}, 3(3):713--770, 1990.

\bibitem[Wei38]{Weil_matrix_div}
A.~Weil.
\newblock G\'en\'eralisation des fonctions ab\'eliennes.
\newblock {\em J. Math. Pur. Appl. (9)}, 17:47--87, 1938.

\bibitem[Yok95]{Yo}
K.~Yokogawa.
\newblock Infinitesimal deformation of parabolic {H}iggs sheaves.
\newblock {\em Internat. J. Math.}, 6(1):125--148, 1995.

\end{thebibliography}
\end{document}